\documentclass{article}
\usepackage{amssymb,amsmath,amsthm,graphicx,cite, float, xcolor}

\def\tr{\triangleright}

\newtheorem{theorem}{Theorem}

\newtheorem{proposition}[theorem]{Proposition}
\newtheorem{corollary}[theorem]{Corollary}

\theoremstyle{definition}
\newtheorem{example}{Example}
\newtheorem{definition}{Definition}
\newtheorem{remark}{Remark}

\title{\Large \textbf{Bilinear Enhancements of Quandle Invariants}}

\date{}

\author{Will Gilroy\footnote{Email: wgilroy@hmc.edu. Supported by Harvey Mudd College's Giovanni Borrelli Mathematics Fellowship.}
\and
Sam Nelson\footnote{Email: Sam.Nelson@cmc.edu. Partially supported by Simons Foundation collaboration grant 702597.}}

\begin{document}
\maketitle

\begin{abstract}
We enhance the quandle counting invariants of oriented classical and virtual
links using a construction similar to quandle modules but inspired
by symplectic quandle operations rather than Alexander quandle operations. 
Given a finite quandle $X$ and a vector space $V$ over a field, sets of 
bilinear forms on $V$ indexed by pairs of elements of $X$ satisfying certain
conditions yield new enhanced multiset- and polynomial-valued invariants of 
oriented classical and virtual links. We
provide examples to illustrate the computation of the invariants and to show that
the enhancement is proper.
\end{abstract}

\parbox{6in} {\textsc{Keywords:} Quandles, enhancements, bilinear enhancements

\smallskip

\textsc{2020 MSC:} 57K12}

\section{\large\textbf{Introduction}}\label{I}

\textit{Quandles} are algebraic structures with axioms encoding
the Reidemeister moves in knot theory, analogous to the way the group axioms 
encode geometric symmetries. Finite quandles $X$ define non-negative integer
valued invariants of oriented links in $\mathbb{R}^3$ via the
set of quandle homomorphisms from the \textit{knot quandle} $\mathcal{Q}(K)$ 
of a knot or link $K$ to $X$. The elements
of the homset can be represented as \textit{colorings} of a diagram of the
knot or link, i.e., assignments of elements of $X$ to the arcs in the diagram,
in a way analogous to representation of linear transformations via matrices. More
precisely, a different choice of diagram for the knot or link will
give a different representation of the the same homset element in the same way
that a different choice of basis yields a different matrix for the same
linear transformation. The cardinality of the quandle homset is then an
integer-valued invariant of the knot or link for each finite quandle $X$,
known as the \textit{quandle counting invariant}, denoted 
$\Phi_X^{\mathbb{Z}}(K)=|\mathrm{Hom}(\mathcal{Q}(K),X)|$.

An \textit{enhancement} of the quandle counting invariant can be defined
whenever we have an invariant $\phi$ of quandle homset elements; then
the multiset of $\phi$-values over the homset defines an invariant whose
cardinality recovers $\Phi_X^{\mathbb{Z}}(K)$ but is generally stronger. In
several previous papers such as \cite{HHNYZ}, the second listed author as 
well as others
used \textit{quandle modules} to enhance the quandle counting invariant. These
are modules over commutative rings with identity which are invariants of 
quandle homset elements, very much like Alexander module invariants of 
the quandle-colored diagrams. Such a module can be interpreted as a set of 
\textit{bead colorings} of a quandle-colored diagram with bead interaction
rules at crossings representing a kind Alexander quandle operation deformed
by the quandle homset colors.

In this paper we pursue a similar idea with role of Alexander quandles replaced
with another type of quandle structure defined on modules over commutative 
rings with identity, namely \textit{symplectic quandles}. The paper is organized
as follows. In Section \ref{QB} we review the basics of quandle theory. In 
Section \ref{BE} we define our \textit{bilinear enhancements} and use them
to introduce an infinite family of polynomial-valued invariants of knots and 
links. In Section \ref{E} we provide examples to illustrate the computation
of the invariants and in particular show that the enhancement is proper, i.e,
not determined by $\Phi_X^{\mathbb{Z}}$. We conclude in Section \ref{Q} with 
questions for future research.

\section{\large\textbf{Quandle Basics}}\label{QB}

We begin with a definition. See \cite{EN,J,M} for more.

\begin{definition}
A \textit{quandle} is a set $X$ with a binary operation $\tr:X\times X\to X$
satisfying the conditions
\begin{itemize}
\item[(i)] For all $x\in X$, $x\tr x=x$,
\item[(ii)] For all $y\in X$, the map $f_y:X\to X$ defined by $f_y(x)=x\tr y$
is invertible, and
\item[(iii)] For all $x,y,z\in X$, we have
\[(x\tr y)\tr z=(x\tr z)\tr(y\tr z).\] 
\end{itemize}
A quandle in which the maps $f_y$ from axiom (ii) are involutions is called
an \textit{involutory quandle} or \textit{kei}. We may write $x\tr^{-1} y$ to
refer to $f_y^{-1}(x)$.
\end{definition}

\begin{example}
Common examples of quandle structures include:
\begin{itemize}
\item Groups $G$ with $g\tr h=h^{-1}gh$, known as \textit{conjugation quandles},
\item Groups $G$ with $g\tr h=hg^{-1}h$, known as \textit{core quandles},
\item Modules over $\mathbb{Z}[t^{\pm 1}]$ with $x\tr y =tx+(1-t)y$, known
as \textit{Alexander quandles}, and
\item Modules $M$ over a commutative ring $R$ with a symplectic form 
$[,]:M\times M\to R$ with $x\tr y=x+[x,y]y$, known as \textit{symplectic 
quandles}.
\end{itemize}
\end{example}

\begin{example}
An important example is the \textit{knot quandle} $\mathcal{Q}(K)$ of an 
oriented knot or link $K$. Given a diagram $D$ of $K$, $\mathcal{Q}(K)$
has a presentation with a generator $a_i$ for each arc in $D$ and a relation
$a_i\tr a_j=a_k$ at each crossing in $D$ as shown.
\[\includegraphics{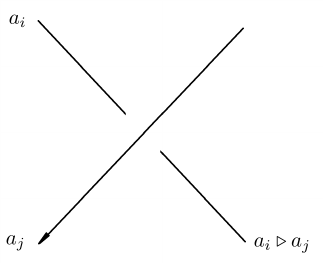}\]
Then the knot quandle is the set of equivalence classes of quandle
words in these generators, i.e. expressions defined recursively as 
generators $a_j$ or expressions of the form $w\tr z$ or $w\tr^{-1} z$ 
where $w,z$ are words, under the equivalence relation generated by the crossing 
relations together with the quandle axioms.
\end{example}

Now, let $X$ be a finite quandle and $D$ an oriented knot or link diagram
representing an oriented knot or link $K$. Then an assignment of elements
of $X$ to each generator $a_i$ in $D$ defines a quandle homomorphism 
$f:\mathcal{Q}(K)\to X$ if and only if the crossing relations $f(a_i)\tr f(a_j)=f(a_k)$ are satisfied at every crossing. Such an assignment is known as an
\textit{$X$-coloring} of $D$. In this way we can represent the quandle 
homset $\mathrm{Hom}(\mathcal{Q}(K),X)$ via $X$-colorings of $D$ and use this
to compute the quandle homset by finding all colorings which satisfy
the coloring condition at every crossing. Let us write $D_f$ to denote the
$X$-coloring of $D$ determined by the homset element $f$.

Performing a Reidemeister move on an $X$-colored diagram results
in a unique $X$-coloring on the diagram following the move; thus, we can 
consider invariants of $X$-colored diagrams for any quandle $X$. An invariant
$\phi$ of $X$-colored diagrams then yields an invariant of links
called an \textit{enhancement} by taking the multiset of $\phi$-values
over the quandle homset $\mathrm{Hom}(\mathcal{Q}(K),X)$. Examples are 
plentiful in the literature, starting with the \textit{quandle cocycle 
invariants} introduced in \cite{CJKLS} and continuing with many more
enhancements such as \textit{quantum enhancements} \cite{N}, 
\textit{quandle module enhancements} \cite{CEGM}, 
\textit{quandle polynomial enhancements} \cite{N2} and many others.

\section{\large\textbf{Bilinear Enhancements}} \label{BE}

We begin with a definition.

\begin{definition} \label{def1}
Let $X$ be a finite quandle and $V$ a vector space over a field $\mathbb{F}$. 
An \textit{$X$-bilinear form} on $V$ is a choice of bilinear form 
$[,]_{x,y}:V\times V\to \mathbb{F}$ for each ordered pair $x,y\in X$
of quandle elements such that for all $x,y,z\in X$ and $a,b,c\in M$ we have
\[\begin{array}{rclr}
{}[a,a]_{x,x} & = & 0 & (i) \\
{}[a,b]_{x,y} & = & [a+[a,c]_{x,z}c,b+[b,c]_{y,z}c]_{x\tr z, y\tr z} & (ii)\\
{}[a,c]_{x\tr y,z} +[a,b]_{x,y}[b,c]_{x\tr y,z} & = & [a,c]_{x,z} + [a,b]_{x,y}[b,c]_{y,z} & (iii)\\
\end{array}\]
\end{definition}

We will generally denote an $X$-bilinear form as $\phi$. If the dimension 
of $V$ is $n$ and $|X|=m$, we can specify such a $\phi$ with an $m\times m$ 
matrix of $n\times n$ matrices $B_{x,y}$ over $\mathbb{F}$ such that for 
$\vec{u},\vec{v}\in V$ and $x,y\in X$ we have
\[{}[\vec{u},\vec{v}]_{x,y}= \vec{u}^TB_{x,y}\vec{v}.\]

\begin{example}
Given any quandle $X$ and $\mathbb{F}$-vector space $V$, we obtain a trivial
example of an $X$-bilinear form by setting $[,]_{x,y}=0$.
\end{example}

\begin{example}
As a (slightly) less trivial example, if $X=V$ is a symplectic quandle with 
symplectic form $[,]$ then setting $[,]_{x,y}=[,]$ gives $X$ the structure of
an $X$-bilinear form.
\end{example}

\begin{example}\label{ex:nt}
Let $X$ be the quandle structure on the set $\{1,2,3\}$ defined by the 
operation table 
\[\begin{array}{r|rrr}
\tr & 1 & 2 & 3 \\ \hline
1 & 1 & 1 & 2 \\
2 & 2 & 2 & 1 \\
3 & 3 & 3 & 3
\end{array}\]
and let $V=(\mathbb{Z}_2)^2$. Then one can check that the array of matrices
\[\left[\begin{array}{rrr}
\left[\begin{array}{rr} 0 & 1 \\ 1 & 0 \end{array} \right] &
\left[\begin{array}{rr} 0 & 1 \\ 1 & 0 \end{array} \right] &
\left[\begin{array}{rr} 0 & 0 \\ 0 & 0 \end{array} \right] \\ \\
\left[\begin{array}{rr} 0 & 1 \\ 1 & 0 \end{array} \right] &
\left[\begin{array}{rr} 0 & 1 \\ 1 & 0 \end{array} \right] &
\left[\begin{array}{rr} 0 & 0 \\ 0 & 0 \end{array} \right] \\ \\
\left[\begin{array}{rr} 0 & 0 \\ 0 & 0 \end{array} \right] &
\left[\begin{array}{rr} 0 & 0 \\ 0 & 0 \end{array} \right] &
\left[\begin{array}{rr} 0 & 0 \\ 0 & 0 \end{array} \right] \\
\end{array}\right]\]
defines an $X$-bilinear form on $V$.
\end{example}

Definition \ref{def1} is motivated by the notion of giving a secondary 
coloring to the arcs in an $X$-colored oriented link diagram with 
vectors pictured as ``beads'' satisfying the \textit{bead coloring rule}:
\[\includegraphics{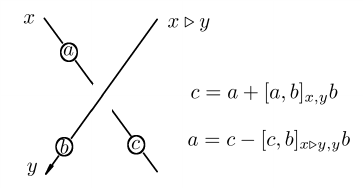}.\]

\begin{definition}
Let $X$ be a finite quandle, $V$ an $\mathbb{F}$-vector space and
$\phi$ an $X$-bilinear form. A \textit{$(X,\phi)$-coloring} of an oriented 
link diagram $D$ is an assignment of an element of $X$ and an element of $V$
to each arc in $D$ such that the bead coloring rule is satisfied at every 
crossing in $D$.
\end{definition} 

\noindent
We can now state our main result.

\begin{proposition}
If an oriented $X$-colored link diagram has $\phi$-coloring before
a Reidemeister move, there is a unique $\phi$-coloring of the diagram
after the move which agrees with the original coloring outside the
neighborhood of the move.
\end{proposition}

\begin{proof}
This is a matter of checking the statement for a generating set of oriented 
Reidemeister moves as in \cite{P}; indeed, the axioms defining $X$-bilinear
forms are chosen precisely with this property in mind.  We apply the above 
\textit{bead coloring rule} to the Reidemeister moves to generate the 
conditions given in Definition \ref{def1}. Let us use
the generating set of Reidemeister moves shown:
\[\includegraphics{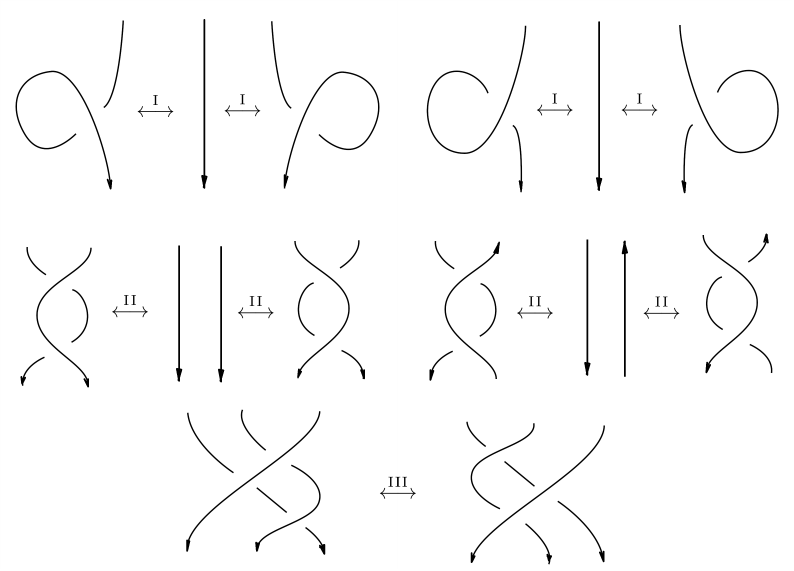}\]

\noindent
Each type I move gives condition (i); we illustrate 
with one of the four:
\[\includegraphics{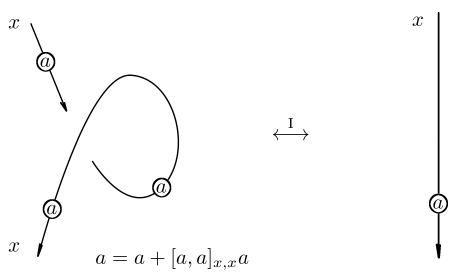}\]

For the type II Reidemeister moves we need to check that the two coloring 
rules are consistent; i.e., we need
\begin{eqnarray*}
a & = & c-[c,b]_{x\tr y, y}b 
 =  a+[a,b]_{x,y}b-[a+[a,b]_{x,y}b,b]_{x\tr y, y}b,
\end{eqnarray*}
so it suffices to show that
\[[a,b]_{x,y}=[a+[a,b]_{x,y}b,b]_{x\tr y,y}.\]
The case of condition (ii) where $y=z$ and $b=c$ implies
\[[a,b]_{x,y}=[a+[a,b]_{x,y}b,b+[b,b]_{y,y}b]_{x\tr y,y\tr y},\]
which, by condition (i) and $y\tr y=y$, gives
\[[a,b]_{x,y}=[a+[a,b]_{x,y}b,b]_{x\tr y,y}\]
as required.

Conditions (ii) and (iii) are required by the all-positive Reidemeister
III moves, the last move in the generating set.

\[\includegraphics{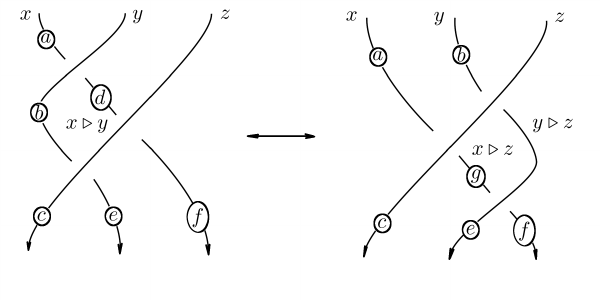}\]

\noindent 
Starting with our bead coloring rule and using bilinearity of $[,]_{x,y}$ we have
\begin{eqnarray*}
d & = & a+[a,b]_{x,y}b \\
e & = & b+[b,c]_{y,z}c \\
f & = & d+[d,c]_{x\tr y,z}c \\
& = & a+[a,b]_{x,y}b+[a+[a,b]_{x,y}b,c]_{x\tr y,z}c \\
& = & a+[a,b]_{x,y}b+[a,c]_{x\tr y,z}c +[a,b]_{x,y}[b,c]_{x\tr y,z}c \\
& = & a+[a,b]_{x,y}b+([a,c]_{x\tr y,z} +[a,b]_{x,y}[b,c]_{x\tr y,z})c. \\
\end{eqnarray*}
Moreover, 
\begin{eqnarray*}
g & = & a+[a,c]_{x,z}c \\
e & = & b+[b,c]_{y,z}c \\
f & = & g+[g,e]_{x\tr z, y\tr z}e \\
 & = & a+[a,c]_{x,z}c+[a+[a,c]_{x,z}c,b+[b,c]_{y,z}c]_{x\tr z, y\tr z}(b+[b,c]_{y,z}c)\\
 & = &  a + Ab + \left( [a,c]_{x,z} + [b,c]_{y,z}A \right)c 
\end{eqnarray*}
where 
\[ A = [a+[a,c]_{x,z}c,b+[b,c]_{y,z}c]_{x\tr z, y\tr z}. \]
Then comparing coefficients in $f$ yields condition (ii) and condition (iii) in 
Definition \ref{def1}.
\end{proof}

\begin{corollary}
Let $X$ be a quandle, $V$ an $\mathbb{F}$-vector space and $\phi$ and 
$X$-bilinear form. Then for any $X$-colored oriented knot or link diagram
$D_f$, the number of $(X,\phi)$-colorings of $D_f$ is unchanged by $X$-colored
Reidemeister moves.
\end{corollary}

\noindent
Let us denote the set of $(X,\phi)$-colorings of $D_f$ by $\mathcal{C}_X^{\phi}(D_f)$.

\begin{definition}
Let $X$ be a quandle, $M$ a module over a commutative ring with identity $R$,
and $\phi =\{[,]_{x,y}\ \vert \: x,y\in X\}$ an $X$-bilinear form on $M$. Given 
an $X$-coloring of a diagram $D$ of an oriented link, we define the 
\textit{$X$-bilinear enhanced multiset} of $D$ to be the multiset of bead 
colorings of $D_f$ 
where $D_f$ ranges over the set of $X$-colorings of $D$,
\[\Phi_X^{\phi,M}(D)=\{\mathcal{C}_X^{\phi}(D_f)\ |\ 
f\in\mathrm{Hom}(\mathcal{Q}(D),X)\}\]
and we define the \textit{$X$-bilinear enhanced polynomial} of $D$ to be
\[\Phi_X^{\phi}(D)=\sum_{f\in\mathrm{Hom}(\mathcal{Q}(D),X)} u^{|\mathcal{C}_X^{\phi}(D_f)|}.\]
Moreover, for any knot or link $L$ we define $\Phi_X^{\phi,M}(L)$ and
$\Phi_X^{\phi}(L)$ to be $\Phi_X^{\phi,M}(D)$ and $\Phi_X^{\phi}(D)$ respectively
where $D$ is any diagram of $L$.
\end{definition}

\noindent
Then, the above shows that we have the following corollary.
\begin{corollary}
For any finite quandle $X$, $\mathbb{F}$-vector space $V$ and $X$-bilinear
form $\phi$, the multiset $\Phi_X^{\phi,M}(L)$ and polynomial 
$\Phi_X^{\phi}(L)$ are invariant under Reidemeister moves and hence invariants
of oriented links.
\end{corollary}

\begin{remark}
We note that evaluating $\Phi_X^{\phi}(L)$ at $u=1$ yields the quandle counting
invariant $|\mathrm{Hom}(\mathcal{Q}(L),X)|$, so the new polynomial invariant 
is an enhancement of the quandle counting invariant. In the next section we will
we show that the enhancement is proper, i.e., not determined by the 
quandle counting invariant in general.
\end{remark}

\begin{remark}
We note also that nothing in the definition of $\Phi_X^{\phi,M}(L)$ or
$\Phi_X^{\phi}(L)$ depends on the supporting surface of having genus zero, so 
both are also well-defined invariants of virtual links.
\end{remark}

\section{\large\textbf{Examples}}\label{E}

In this section we collect a few examples of the new invariants.

\begin{example}
Let $\mathbb{F}=\mathbb{Z}_2$ and $V=\mathbb{F}^2$. Let $X$ and $\phi$ be the 
quandle and $X$-bilinear form given in Example \ref{ex:nt}. Then the pictured 
Hopf link $L2a1$ has five $X$-colorings as shown:
\[\includegraphics{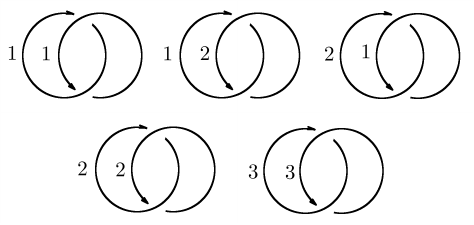}\]
Then putting beads on each arc, we have bead-coloring equations for the first
$X$-coloring
\[
\raisebox{-0.35in}{\includegraphics{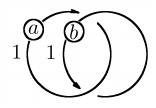}}
\quad
\begin{array}{rcl}
a & = & a+[a,b]_{1,1}b \\
b & = & b+[b,a]_{1,1}a
\end{array}\]
Writing $a=(a_1,a_2)$ and $b=(b_1,b_2)$, these become
\[\begin{array}{rcl}
a_1 & = & a_1+(a_1b_2+a_2b_1)b_1 \\
a_2 & = & a_2+(a_1b_2+a_2b_1)b_2. \\
\end{array}\]

We solve these equations to find the number of valid bead colorings for this $X$-coloring.
If $b=(0,0)$ all four choices of $a\in(\mathbb{Z}_2)^2$ are solutions;
if $b=(1,1)$ we need $a=(1,1)$ or $a=(0,0)$; if $b=(1,0)$ we have solutions
$a=(1,0)$ and $a=(0,0)$ and if $b=(0,1)$ then $a=(0,1)$ and $a=(0,0)$ are 
solutions. Thus we have 10 bead colorings for this $X$-coloring and so its 
contribution to the enhanced polynomial is $u^{10}$.

A similar analysis reveals three more contributions of $u^{10}$ by other $X$-colorings; however, the
monochromatic $X$-coloring by the quandle element $3$ has bead coloring equations
$a=a$ and $b=b$, so there are 16 total bead colorings and this $X$-coloring
contributes $u^{16}$. Hence, we obtain invariant value
\[\Phi_X^{\phi}(L2a1)=u^{16}+4u^{10}.\]
\end{example}

\begin{example}\label{ex:2}
Using \texttt{python}, we computed $\Phi_X^{\phi}(L)$ for a choice of 
orientation for each of the prime
links with up to seven crossings as found at the knot atlas \cite{KA} with 
respect to the quandle $X$ and $X$-bilinear form on $V=(\mathbb{Z}_2)^2$
in Example \ref{ex:nt}; the results are in the table.

\[\begin{array}{r|l}
\Phi_X^{\phi}(L) & L \\ \hline
u^{16}+4u^{10} & L2a1, L6a2, L7a6 \\
5u^{16} & L6a3, L7a5 \\
2u^{40}+7u^{16} & L7a2, L7a3, L7n1, L7n2 \\
5u^{16}+4u^{10} & L4a1, L5a1, L7a4 \\
9u^{16} & L6a1, L7a1 \\
7u^{64}+8u^{22} & L6n1, L7a7 \\
7u^{64}+8u^{28} & L6a5 \\
19u^{64}+8u^{40} & L6a4
\end{array}
\]
In particular, this table shown several examples of links with the same number 
of $X$-colorings distinguished by the bead-coloring information, and we have 
shown that the enhancement is proper.
 \end{example}

\begin{example}
For our final example, we computed the invariant for the same set of links with the same quandle as in Example \ref{ex:2} but with a different choice of 
bilinear enhancement over the same vector space, namely
\[\phi=\left[\begin{array}{ccc}
\left[\begin{array}{rr} 0 & 1 \\ 1 & 0 \end{array}\right] & 
\left[\begin{array}{rr} 0 & 1 \\ 1 & 0 \end{array}\right] &
\left[\begin{array}{rr} 0 & 0 \\ 0 & 0\end{array}\right] \\ \\
\left[\begin{array}{rr} 0 & 1 \\ 1 & 0\end{array}\right] &
\left[\begin{array}{rr} 0 & 1 \\ 1 & 0\end{array}\right] &
\left[\begin{array}{rr} 0 & 0 \\ 0 & 0\end{array}\right] \\ \\
\left[\begin{array}{rr} 0 & 0 \\ 0 & 0\end{array}\right] &
\left[\begin{array}{rr} 0 & 0 \\ 0 & 0\end{array}\right] &
\left[\begin{array}{rr} 0 & 1 \\ 1 & 0\end{array}\right] \end{array}\right].\]

The results are collected in the table.
\[\begin{array}{r|l}
\Phi_X^{\phi}(L) & L \\ \hline
5u^{10} & L2a1, L6a2, L7a6 \\
5u^{16} & L6a3, L7a5 \\
4u^{16}+5u^{10} & L4a1, L5a1, L7a4 \\
4u^{40}+5u^{16} & L7a2, L7a3, L7n1, L7n2 \\
9u^{16} & L6a1, L7a1 \\
6u^{40}+9u^{22} & L6n1, L7a7 \\
6u^{40}+9u^{28} & L6a5 \\
18u^{64}+9u^{40} & L6a4 
\end{array}\]

 \end{example}

\section{\large\textbf{Questions}}\label{Q}

We end with some questions for future research.

The bilinear enhancement idea could be generalized in a number of ways, 
all of which could provide interesting new families of invariants. What 
extra conditions would be needed to replace the understand bead coefficient
with a second bilinear form depending on the quandle colors? How about 
generalizations to other knot-theoretic categories such as spatial graphs,
handlebody-knots or surface-links? Alternatively, we could replace the
quandle colorings with invariant colorings by other algebraic structures
such as groups, biquandles or tribrackets.

As with all such families of invariants, it is interesting to ask what 
precise information about the knot or link is being extracted by these
invariants. For example, in \cite{HN} the quandle colorings of a two-component
link by a specific family of quandles were shown to be measuring the linking 
number of the link. This and many of other families of coloring invariants are 
waiting for similar connections with other invariants and properties of knots 
and links to be found.

\bibliography{wg-sn}{}
\bibliographystyle{abbrv}

\bigskip

\noindent
\textsc{Department of Mathematical Sciences \\
Claremont McKenna College \\
850 Columbia Ave. \\
Claremont, CA 91711}

\end{document}